\documentclass{article} 
\usepackage[english]{babel}
\usepackage{graphicx}
\usepackage{subcaption}
\usepackage{overpic}
\usepackage{amsthm}
\newtheorem{theorem}{Theorem}

\begin{document}
\date{} 
\title{Stability of solutions of systems  of Volterra integral equations}
\author{Ilya Boykov $^{1}$, Vladimir Roudnev  $^{2}$ and Alla Boykova  $^{3}$}
\maketitle
\abstract{In this paper we  propose new sufficient conditions for stability of solutions of systems of Volterra linear  integral equations and systems of linear integro-differential Volterra equations. 
Solution stability conditions for systems of  Volterra linear   integral equations are studied with perturbed right hand sides of the equations.  
These new sufficient conditions are expressed in terms of logarithmic norms of matrices consisting of coefficients of equations and derivatives of their kernels. 
For integro-differential Volterra equations we also provide sufficient conditions for stability of solutions when the initial conditions  are perturbed.
}
\footnotetext[1]{Penza State University; boikov@pnzgu.ru; corresponding author .}
\footnotetext[2]{Saint Petersburg State University; v.rudnev@spbu.ru }
\footnotetext[3]{Penza State University; aiboikova@pnzgu.ru }

\section{Introduction}

Volterra's integro-differential equations have multiple applications,  such as the  problems of  wing movement in  non-stationary air flow
 \cite{Bel_1971, Bel_1980}, including turbulent flows, mechanics of solids, including viscoelastic properties \cite{Arut}.

The studying of the Volterra integro-differential equations originates from works of Volterra \cite{Volterra_1910, Volterra_1912}. On the base of such equations he had been studying a number of models that take into account the heredity effects \cite{Volterra_1959}, suggested a mathematical model of biological populations dynamic based on the use of integro-differential equations. The effects of heredity are particularly important in modeling immune processes in living organisms, where Volterra integro-differential equations play a major role. 

 Here we refer to Guardiola and Vecchio \cite{Guardiola} who proposed  the immune response model with distributed delay  
\begin{equation}
	\label{(1.10)} 
	\begin{array}{llll}
		P(t) = \int\limits^t_0 F(t-x) e^{-\delta(t-x)} k_i k_s V(x)S(x) dx,\\
		V(t) = V_0 e^{-ct} + \int\limits^t_0  e^{-c(t-x)}  pP(x) dx, \, t \in [0, t_f],\\
		S(t) = S_0 e^{-\beta t} + \int\limits^t_0  e^{-\beta(t-x)} [\alpha - k_s V(x) S(x)] dx,
	\end{array}
\end{equation}
where $P$ is the population of infected cells that produce viruses, $V$ is the population of viruses,
$S$ is the population of cells susceptible to the infection or target cells. The function $F$ is a continuously distributed intracellular delay, 
$$
F(x)=\int\limits_0^xf(s)ds=1-e^{-x/b}\sum\limits_{j=0}^{n-1}\frac{x^j}{j!b^j},\ \ f(x)=\frac{x^{n-1}}{(n-1)!b^n}e^{-x/b}, \ \ \ \int\limits_0^\infty xf(x)dx=nb,
$$
$\alpha>0 $ is the renewal rate of susceptible cells, 
$\beta $ is the clearance rate of susceptible cells, 
$k_s $ is the rate of infection, 
$S_0 \ge 0$, $c> 0 $ is the clearance rate of viruses, 
$p>0 $ is the virus production rate, 
$V_0 \ge 0$,  $\delta > 0 $ is the clearance rate of productive cells,
$k_i>0 $ is the rate of transition between infected and productively infected cells, 
$b\in R^+, n\in N $ are the  parameters of the Gamma probability distribution.

Guardiola and Vecchio \cite{Guardiola} have performed the analysis of the qualitative behavior of the solution of system (\ref{(1.10)}) 
and demonstrated its positivity and boundedness.

In particular stationary state of  biological population consisting of species that survive on a single limited source of food is described with a system of Volterra integro-differential equations. The model of such a population was given by  V. Volterra  \cite{Volterra_1959}.

In scientific literature there are many definitions of stability of solving systems of Volterra linear integral equations. Here is a definition the closest to that one used below.

Consider the system 
\begin{equation}
	X(t)=\int_a^tQ(t,s)X(s)ds+F(t), \  t\geq a,
	\label{1_1}
\end{equation}
where $X(t), F(t)$ is a continuous mapping $[a,\infty)$ in $R^n$; $n\times n$ is a matrix  $Q(t,s)$ satisfied the Carathéodory condition in $a\leq s\leq t$. 

{\bf Definition 1.}  \cite{Tsalyuk_1968}.  The system (\ref{1_1}) or the kernel  $Q(t,s)$ is  stable if for any $\varepsilon >0$ there exists $\delta>0$ that $\sup\limits_{t>a}\|F(t)\|<\delta$ implies
$\sup\limits_{t>a}\|X(t)\|<\delta$.
The system (\ref{1_1}) is asymptotically stable, if it is stable and  $\lim\limits_{t\rightarrow \infty}\|F(t)\|=0$ implies $\lim\limits_{t\rightarrow \infty}\|X(t)\|=0$.

Various definitions of stability solutions of systems of Volterra integral equations are given in  \cite{Tsalyuk_1977}. 

Numerous works are devoted to study of stability for  solutions of Volterra integral equations and their systems. 

Vinokurov   \cite{Vin_1959}   shown that solution of the system of equations 
\begin{equation}
	X(t)+\int_0^th(t,\tau)X(\tau)d\tau=F(t), X(t), F(t) \in R^n,
	\label{v1}
\end{equation}
is stable if and only if, 
\[
\sup\limits_{t\geq 0}\int_0^t\|R(t,s)\|ds<\infty,
\]
where $R(t,s)$ is the resolvent of the system of equations (\ref{v1}). 

In a series of papers, for instance in \cite{Vin_1973}, necessary and sufficient conditions for stability of Volterra integral equations  are expressed via system fundamental matrix. 

Tsalyuk \cite{Tsalyuk_1968} obtained the sufficient condition for stability using  modulus of eigenvalues of iterated matrices.

A detailed review of works dedicated to  solution stability for systems of Volterra integral equations and published up to 1976 was given in  \cite{Tsalyuk_1977}.

In 1979  Corduneanu and  Lakshmikantham  \cite{Corduneanu and  Lakshmikantham} published the survey devoted to problems of existence, uniqueness and stability solutions of integro-differential equations with infinite delays involving Volterra equations. In 2000 Cordunearu  \cite{Corduneanu} published his survey devoted to abstract Volterra equations.  

Stability of solutions of systems of Volterra  integral equations 
\[X(t)=\int_a^t h(t,\tau)X(\tau)d\tau+F(t)\]
has been investigated in  \cite{Bownds}.
Here $h$ is a $n\times n$ matrix which is assumed continuous for $t_0\leq \tau \leq t<\infty, a>t_0$, and $F$, $X$ are $n$-vectors which are at least continuous on $[t_0,\infty)$
assuming that  $h(t,\tau)$ is a degenerate kernel  (Pincherle-Coursat kernel).

In   \cite{Gil} the following system of equations has been studied
\begin{equation}
	X(t)-\int\limits^t_0 K(t-s) X(s) ds+ G(X, t)=F(t)\geq 0,
	\label{09}
\end{equation}
where $K(t)$ is $n\times n$ matrix function, $X$ and $F$ are  $n$-dimensional vectors. 

Recently investigations of stability  of mild solution for a class of impulsive neutral stochastic functional-differential equations in Hilbert spaces  \cite{Deng}   have been actively developed alongside  the classical (deterministic) Volterra equations and  systems of classical Volterra equations.

Stability of solutions of integro-differential Volterra  equations  has been broadly studied for the last fifty years.  
Integro-differential equations
\begin{equation}
\frac{dX(t)}{dt}=A(t)X(t)+\int_0^tC(t-s)X(s)ds, 0\leq t<\infty,   
\label{01}
\end{equation}
and
\begin{equation}
\frac{dX(t)}{dt}=A(t)X(t)+\int_0^tC(t,s)X(s)ds, 0\leq t<\infty,   
\label{02}
\end{equation}
have been investigated. Here $A$ is a  $n\times n$ continuous matrix for $t \in [0,\infty)$, and $C$ is a  $n\times n$ continuous matrix for $0\leq s \leq t \leq \infty$.

Equation (\ref{01}) is a partial case of equation (\ref{02}), and the results obtained in our work for equation (\ref{02}) are equally valid for equation (\ref{01}).
In literature, however, equations (\ref{01}) and (\ref{02}) are considered separately on the base of quite distinct methods. To study stability for the equation (\ref{01}) the  method of norm estimations for operator-valued functions  was introduced in \cite{Gil}. Stability of solutions of systems of integro-differential equations   (\ref{02}) 
was investigated by the second Lyapunov method in \cite{Burton}. 

In \cite{Andreev_2018} Lyapunov-Krasovsky functionals have been  constructed to study the system of equations 
\[\frac{dX(t)}{dt}=f_1(t,X(t))+f_2(t,X(t-\mu_1(t)))
+\int_{t-\mu_2(t)}^tq_1(t,s,X(s))ds+\int_{t_0}^tq_2(t,s,X(t),X(s))ds,\] 
\[\ t_0\leq t \leq T,\]
where $X \in R^n$, $R^n$ is the $n$-dimensional linear real space. The conditions for asymptotic stability have also been obtained.  

The system of Volterra integro-differential equations 
\[
\frac{dX(t)}{dt}=A(t)f(X(t))+\int_0^tB(t,\tau)g(x(\tau))d\tau,
\]
has been studied using the second Lyapunov method. 
Here $A(t)$ is a $n\times n$ functional matrix continuous on $[0,\infty)$, $B(t,s)$ is a $n\times n$ matrix continuous for $0\leq s\leq t<\infty$, $f$ and $g$ are $n\times 1$ vector functions defined on $(-\infty, \infty)$ \cite{Vanualailai}.

Khvorost and  Tsalyuk \cite{Khvor}  investigated the stability of solutions of the system of integro-differential equations of the type 
\begin{equation}
	\frac{dX(t)}{dt}=AX(t)+\int_0^t K(t-s)X(s)ds+F(t,X)+\int_0^tG[t,s,X(s)]ds, 
	\label{2_1}
\end{equation}
where $A$ is a constant  $n\times n$ matrix, $K \in L_1 [0,\infty)$, $F$ and $G$ are continuous for  $0\leq s\leq t<\infty$, moreover $F(t,0)\equiv 0, G(t,s,0)\equiv 0$.

The stability for solution of the system of equations (\ref{2_1}) was proved assuming, that 
\[
\sup\limits_{t}\|F(t,X)\|=o(\|X\|), \|X\| \rightarrow 0
\]
and
\[
\sup\limits_{t}\int_0^t\sup\limits_{\|X\|\leq r}\|G(t,s,X)\|ds=o(r), r\rightarrow 0, 
\]
and for similar conditions.

Andreev and  Peregudova  \cite{Andreev_2018}  investigated Volterra integral equations applications for nonlinear controller construction.
Using the second Lyapunov method they studied the  asymptotic stability of solution for  
\[
\frac{dX(t)}{dt}=F(X(t))+\int\limits_0^tG(s-t,X(t),X(s))ds.
\]

Andreev an Peregudova in \cite{Andreev_2021}  studied asymptotic stability of solution of Volterra nonlinear   integro-differential equations with the second Lyapunov method 
\[
\frac{dX(t)}{dt}=F(X(t),\int\limits_{-\infty}^t C(t,s,X(s))ds)
\]
assuming $F(t,0,0)\equiv 0, G(t,0,0)\equiv 0$.

In \cite{Sergeev_2007} the stability of systems of integro-differential equations   
$$
\frac{dX(t)}{dt}=A(t)+\int\limits_0^t K(t,s)X(s) ds+F(X,Y,Z,t)+\mu \Phi(\mu,X,Y,Z,t),
$$
has been investigated. Here $X,Y,Z\in R^n, X=col(x_1,...,x_n)$, $Y=col(y_1,...,y_n)$, $Z=col(z_1,...,z_n)$, $Z$ is analytical function defined by Frechet series , $\mu$ is a small parameter,    $\mu\Phi$ is a perturbation.

Alongside with stability,  Sergeev \cite{Sergeev_2007} studied instability of systems like
$$
\frac{dX(t)}{dt}=A(t)+\int\limits_{t_0}^t K(t,s)X(s) ds+F(X,Y,Z,t)
$$
in critical cases. 

Stability of  Volterra  integro-differential equations with  differential operators of the fourth order is investigated in \cite{Iskan_2002}, \cite{Iskan_2020}.
In particular, they considered the equation 
\[
\frac{d^4 x(t)}{dt^4}+a_3(t)\frac{d^3 x(t)}{dt^3}+a_2(t)\frac{d^2x(t)}{dt^2}
+a_1(t)\frac{dx(t)}{dt}+
\]
\[
+a_0(t) x(t)+\int\limits_0^t \sum\limits^3_{l=0} h_l(\tau,\tau) x^{(l)}(\tau) d\tau=f(t), t>0.
\]

To investigate stability and asymptotic stability of solution of system of integro-differential Volterra equations 
\begin{equation}
\frac{dX(t)}{dt}=G(t,X(t))+\int\limits_0^tH(t,\tau,X(\tau))d\tau,\ \ \ X(0)=X_0
\label{vv1}
\end{equation}
Crisci et al.  \cite{Crisci} employed 
  a method that combines the direct Lyapunov method with the spectral properties of the Jacobian of  the vector function $G(t,X)$.  This method had been further developed in \cite{Vanualailai}, \cite{Tunc}, \cite{Tunc_2022} and applied to more complex equations. Moreover,  the authors in \cite{Tunc}, \cite{Tunc_2022} studied systems of integro-differential Volterra equations with delays and with fractional derivatives.


In this paper we provide an operator method to investigate stability of solutions 
of systems of linear integral and integro-differential Volterra equations. 
We obtained conditions for stability  of solutions  in terms of  logarithmic norms of matrices 
that consist -- for systems of Volterra integro-differential equations -- of equation coefficients and  integral operator kernels. 
In case of systems of Volterra integral  equations  the stability conditions are expressed in terms of equation coefficients  and derivatives of integral operator kernels.

Note the difference in formulations of the stability problems for solving the systems of Volterra integral equations and the systems of Volterra integro-differential equations. 
In the first case one gives right hand side perturbations, in the second one we have perturbations of the initial values. 

The motivation for conducting the study is the following.
We believe that to find sufficient conditions for stability of the systems expressed through coefficients and kernels of equations is an important problem. Similar conditions for stability 
are known for systems of differential equations \cite{Boy_1990}, \cite{Boy_2008},  \cite{Boy_2018}. It seems to be interesting to extend the results to systems of Volterra integral and integro-differential equations. 

The paper consists of five parts. In Introduction we gave a brief review of works devoted to issues of stability for  systems 
of integral and integro-differential Volterra equations.

 We also  recall the definition of the logarithmic norm.
In the second section we
study stability for solutions of systems of linear integral Volterra equations. 
The third  section is devoted to investigation of stability for solutions of systems of integro-differential Volterra equations. In Conclusions we summarize the results of our work.

\subsection{Notation}

Recall the logarithmic norm definition  \cite{Dal}. 

The following notation will be used:
\[B(a,r) = \{z \in B:\|z-a\| \leq r\}\ ,\]
\[ S(a,r) = \{z \in B:\|z-a\|=r\}\ ,\]
\[Re (K) = \Re (K) = (K+K^*)/2\ ,\] 
\[\Lambda(K) = \lim\limits_{h \downarrow 0}(\|I+h K\|-1)h^{-1}\ .\]
Here $B$ is a Banach space, $a \in B$, $K$ is a linear and bounded operator on $B$, $\Lambda(K)$ is the logarithmic norm \cite{Dal} of the operator $K$, $K^*$ is the conjugate operator to $K$, and $I$ stands for the identity operator. 

The logarithmic norms of linear operators are known for the most frequently used spaces.

Let $A = \{a_{ij}\}$, $i,j = 1,2,\ldots,n$, be a real matrix. 
In the $n$-dimensional space $R^n$ of vectors $X=(x_1,\ldots,x_n)$
the following norms are often used: 
 \[\|X\|_1 = \sum\limits^n_{i=1}|x_i|;\]
 \[\|X\|_2 = \max\limits_{1 \leq i \leq n} |x_i|\ ;\]
 \[\|X\|_3 = (\sum\limits^n_{i=1}x_i^2)^{1/2}\ .\]
In addition, it is known that a linear combination with positive coefficients of any finite number of norms is a norm as well.

Here are some expressions of the logarithmic norm of a matrix $ A = (a_{ij})$, due to the above norms of the vectors:\\
logarithmic norm $\Lambda_1$
\[
\Lambda_1(A) = \max\limits_{1 \leq j \leq n}\left(a_{jj} + \sum\limits_{i \ne j}|a_{ij}|  \right);
\]
 logarithmic norm $\Lambda_2$
\[
\Lambda_2(A) = \max\limits_{1 \leq i \leq n}\left(a_{ii} + \sum\limits_{j \ne i}|a_{ij}|  \right);  
\]
 logarithmic norm  $\Lambda_3$
\[
\Lambda_3(A) = \lambda_{\max} \left(\frac{A+A^*}{2}\right),
\]
where $A^*$  is the conjugate matrix for $A$.

Main properties of the logarithmic norm are given in \cite{Dal}.

\section{Stability of solutions of systems  of Volterra linear equations}

Consider a system of Volterra integral equations
\begin{equation}
x_k(t)=\sum\limits^n_{l=1}\int\limits^t_0 b_{kl}(t,\tau) x_l(\tau)d\tau, \   k=1,2,...,n, \ \  0\le t< \infty \ ,
\label{1}
\end{equation}
where $x_k(t)$ are the components of the vector-valued solution, and $b_{kl}(t,\tau)$ are the elements of the integral matrix kernel.

Assuming that the functions $b_{kl}(t,\tau), k,l=1,2,...,n$ are differentiable with respect to the first variable, the system of equations  (\ref{1}) can be written in equivalent form 
\begin{equation}
\frac{dx_k(t)}{dt}=\sum\limits^n_{l=1} b_{kl} (t, t) x_l(t)+
\sum\limits^n_{l=1} \int\limits^t_0 \left(\frac{d}{dt} b_{kl} (t,\tau)\right) x_l(\tau) d\tau, \ k,l=1,2,...,n,
\label{2}
\end{equation}
\begin{equation}
x_k(0)=0, \ k=1,2,...,n.
\label{3}
\end{equation}

We associate the system of equations (\ref{1}) with the perturbed system  
\begin{equation}
x_k(t)=\sum\limits^n_{l=1}\int\limits^t_0 b_{kl}(t,\tau) x_l(\tau)d\tau+f_k(t),
\label{4}
\end{equation}
where the function  $F(t)=(f_1(t),...,f_n(t))$ is a continuous differentiable vector function.

In turn, we associate the equation  (\ref{4}) with the Cauchy problem 
\begin{equation}
\begin{array}{lll}	
\frac{dx_k(t)}{dt}=\sum\limits^n_{l=1} b_{kl} (t, t) x_l(t)+
\\
+\sum\limits^n_{l=1} \int\limits^t_0 \left(\frac{d}{dt} b_{kl} (t,\tau)\right) x_l(\tau) d\tau+f'_k(t), k=1,2,...,n.
\end{array}
\label{5}
\end{equation}
\begin{equation}
x_k(0)=0, k=1,2,...,n.
\label{6}
\end{equation}
Here  $f_k(0)=0$, $k=1,2,...,n$.

If $f_k(0)\neq 0$ for some  $k, k=1,2,...,n, $ we arrive at  the Cauchy problem 
\begin{equation}
\begin{array}{lll}
\frac{dx_k(t)}{dt}=\sum\limits^n_{l=1} b_{kl} (t, t) x_l(t)+
\\
+\sum\limits^n_{l=1} \int\limits^t_0 \left(\frac{d}{dt} b_{kl} (t,\tau)\right) x_l(\tau) d\tau+f'_k(t), k=1,2,...,n.
\end{array}
\label{7}
\end{equation}

\begin{equation}
x_k(0)=f_k(0), k=0,1,...,n.
\label{8}
\end{equation}

Let  $X$  stand for a space of vector-functions  $X(t)=(x_1(t),...,x_n(t))$ with the norm  $\|X(t)\|=\max\limits_{1\leq k \leq n} |x_k(t)|$.

{\bf Definition 2}.  A trivial solution of the system of Volterra  integral equations  (\ref{1}) is stable
with respect to perturbations of $f_k(t), f_k(0)=0, k=1,2,\ldots,n$
 if for any arbitrary small  $\varepsilon(\varepsilon >0)$ there exists such 
$\delta(\varepsilon)$, $\delta(\varepsilon)>0$ that    $\sup\limits_{t\in[0,\infty]} \|F'(t)\|\leq \delta(\varepsilon)$   implies  $\sup\limits_{t\in[0,\infty]} \|X(t)\|\leq \varepsilon$.
Here $ X(t)$ stands for the solutions of the system  (\ref{4}), $F(t)=(f_1(t),\ldots,f_n(t))$.

The stability of solution is defined  differently when the condition $ f_k(0)=0, k=1,2,\ldots,n$ is not imposed on perturbation of  right hand sides of Volterra integral equations. In this case the original system of Volterra integral equations  (\ref{1}) is transformed to the Cauchy problem  (\ref {7}), (\ref {8}).

{\bf Definition 3}.  A trivial solution of the system of Volterra  integral equations  (\ref{1}) is stable
if for any arbitrary small  $\varepsilon(\varepsilon >0)$ there exists such 
$\delta_i(\varepsilon)$, $\delta_i(\varepsilon)>0, i=1,2$, $\delta_2(\varepsilon)\le \varepsilon$, that   $\sup\limits_{t\in[0,\infty]} \|F'(t)\|\leq \delta_1(\varepsilon)$, $\|X(0)\|=\|F(0)\|\le \delta_2(\varepsilon)$ implies $\sup\limits_{t\in[0,\infty]} \|X(t)\|\leq \varepsilon$.

Another definition of solution stability of system of Volterra integral equations is given in 
  \cite{Gil}. Consider the system of equations  (\ref{09}).
Let the following condition is fulfilled for any $X(t)\in L_2([0,\infty)$
\begin{equation}
	\|G(X(t),t)\|_{L_2}\leq q\|X(t)\|_{L_2},
	\label{V1}
\end{equation}
where $q$ is a positive constant.
The definition for stability of solutions for systems of Volterra equations in the metric of 
$L_2(0,\infty)$ has been provided in \cite{Gil}. 

{\bf Definition 4}.  The equation  (\ref{09}) is called stable in the class of functions $G$ satisfying (\ref{V1})  if there exists a constant $\eta(q)$ independent from  
a particular form of $G$ ( but depends on  $q$) such that $\|X\|_{L_2}\leq\eta(q)\|F\|_{L_2}$ for each solution of the equation   (\ref{09}). 

Let us investigate stability of a trivial solution of the system of equations (\ref{1})  assuming that the perturbation satisfies $f_k(0)=0$, $k=1,\ldots,n$.

The equation  (\ref{1})  could be rewritten as the Cauchy problem  (\ref{5}), (\ref{6}).

We need to show that for any arbitrary $\varepsilon>0$  there is such $\varepsilon_0$ that if   $\sup_{t\in [0,\infty)}|f_k{}'(t)|\le \varepsilon_0, k=1,2,\ldots,n$ then
$|x_k(t)|\le \varepsilon, k=1,2,\ldots,n$, for $t\in [0,\infty)$. 

In the matrix form the system  (\ref{5}) reads
\begin{equation}
\frac{dX(t)}{dt}=B(t)X(t)+\int\limits_0^tC(t,\tau)X(\tau)d(\tau)+F'(t),
\label{10}
\end{equation}
where $B(t)=\{b_{kl}(t,t)\},  k,l=1,2,\ldots,n$; $C(t,\tau) =\{c_{kl}(t,\tau)\}, c_{kl}(t,\tau)=\frac{d}{dt}b_{kl}(t,\tau),  k,l=1,2,\ldots,n$; $F'(t)=(f_1{}'(t)),\ldots,f_n{}'(t))$.

We will find the conditions for $\varepsilon_0$  such that  the inequality is fulfilled  
\begin{equation}
\sup_{t\in [0,\infty)}\|X(t)\|\le \varepsilon.
\label{11}
\end{equation}

\begin{theorem}
Let	 $\varepsilon(\varepsilon>0)$ be a sufficiently small positive number.
Let $\sup_{[0,\infty)}\|F'(t)\|\le \varepsilon_0\le \varepsilon$. Let for any $t, t' \in [0,\infty), t>t'$ the inequality be fulfilled 
\begin{equation}
\Lambda(B(t))+\frac{1}{t-t'}\int_{t'}^t\psi^*(\tau)d\tau<0,
\label{12}
\end{equation}	
where 
\[\psi^*(s)=\int_0^s \|C(s,\tau) \|d\tau +\frac{\| F'(s) \|}{\varepsilon}.
\] 
Then the trivial solution of the system  (\ref{1}) is stable and $\| X(t)\| \le \varepsilon, t \in [0,\infty )$. 
\label{Theorem1}
\end{theorem}

\begin{proof}[Proof of Theorem 1]
We will prove Theorem \ref{Theorem1} by contradiction. Suppose that at the moment  $t^*$ the condition $\|X(t)\|\le \varepsilon, t \in [0,\infty)$  is violated. 

The solution of the Cauchy problem  (\ref{10}), (\ref{6}) for $t\ge t^*$ can be presented in the form \cite{Boy_2008,Dal}
\begin{equation}
\begin{array}{lll}
X(t)= \exp\{B(t^*)(t-t^*)\}X(t^*)+\int\limits_{t^*}^t\exp\{B(t^*)(t-s)\}(B(s)-B(t^*))X(s)ds+\\
+\int\limits_{t^*}^t\exp\{B(t^*)(t-s)\}(\int\limits_0^s C(s,\tau)X(\tau)d\tau)ds+  \int\limits_{t^*}^t\exp\{B(t^*)(t-s)\}F'(s)ds.
\end{array}
\label{13}
\end{equation}

Proceeding to norms in  (\ref{13}), we have
\begin{equation}
\begin{array}{lll}
\|X(t)\|\le\exp\{\Lambda{(B(t^*))}(t-t^*)\}\|X(t^*)\|+\\
\int\limits_{t^*}^t \exp\{\Lambda{(B(t^*))}(t-s)\} 
\|B(s)-B(t^*)\|\|X(s)\|ds+ \\
\int\limits_{t^*}^t\exp\{\Lambda{(B(t^*))}(t-s)\}(\int\limits_0^s\|C(s,\tau)X(\tau)\|d\tau)ds+ \\ \int\limits_{t^*}^t\exp\{\Lambda{(B(t^*))}(t-s)\}\|F'(s)\|ds.
\end{array}
\label{14}
\end{equation}

Let $[t^*, t^{**}]$  stands for an interval where the function $\|X(t)\|$ increases. The interval exists since $\|X(t^{*})\|= \varepsilon$ and for $t:\ t^{*}<t<t^{**}$ the function $\|X(t)\|$ increases. 

Then the inequality 
\[\int\limits_0^s\|C(s,\tau)\|\|X(\tau)\|d\tau \le \|X(s)\|\int\limits_0^s\|C(s,\tau)\|d\tau, t^*\le s\le t^{**}
\]
is satisfied.

Let us exploit the inequality (\ref{14}):
\begin{equation}
\begin{array}{lll}
\|X(t)\|\leq \exp\{\Lambda{(B(t^*))}(t-t^*)\}\|X(t^*)\|+\\
\int\limits_{t^*}^t\exp\{\Lambda{(B(t^*))}(t-s)\}\|B(s)-B(t^*)\|\|X(s)\|ds+ \\
\int\limits_{t^*}^t\exp\{\Lambda{(B(t^*))}(t-s)\}\left(\int\limits_0^s\|C(s,\tau)\|\|X(s)\|d\tau\right)ds+ \\ +\int\limits_{t^*}^t\exp\{\Lambda{(B(t^*))}(t-s)\}\frac{\|F'(s)\|}{\varepsilon}\|X(s)\|ds.
\end{array}
\label{15}
\end{equation}

By introducing a function $\varphi(t)=\exp\{-\Lambda{(B(t^*))}t\}\|X(t)\|$
we rewrite the inequality  (\ref{15})  in the form 
\begin{equation}
\begin{array}{rcl}
\varphi(t)&\leq&\varphi(t^*)+\varepsilon_1\int\limits_{t^*}^t\varphi(s)ds+\int\limits_{t^*}^t(\int\limits_{0}^s\|C(s,\tau)\|d\tau)\varphi(s)ds+\int\limits_{t^*}^t
\frac{\|F'(s)\|}{\varepsilon}\varphi(s)ds=\\
&=&\varphi(t^*)+\int_{t^*}^t\left[\varepsilon_1+\left(\int_0^s\|C(s,\tau)\|d\tau\right)+\frac{\|F'(s)\|}{\varepsilon}\right]\varphi(s)ds=\\
&=&\varphi(t^*)+\int_{t^*}^t\psi(s)\varphi(s)ds,
\end{array}
\label{16}
\end{equation}
where $\psi(s)=\varepsilon_1+\left(\int_0^s\|C(s,\tau)\|d\tau\right)+\frac{\|F'(s)\|}{\varepsilon}, \varepsilon_1=\max_{s\in [t^*, t^{**}]} \|B(s)-B(t^*)\|$.

From the last inequality and the  Gronwall-Bellman inequality we have 
\[
 \varphi(t)\leq \varphi(t^*)\exp\left\{\int_{t^*}^t\psi(\tau)d\tau\right\}.
\]

Then
$$\|X(t)\|\leq \exp\left\{\int_{t^*}^t\psi(\tau)d\tau\right\}e^{\Lambda(B(t^*))(t-t^*)}\|X(t^*)\|\ . $$

Therefore, if for  $t \in (t^*,t^{**}]$
$$\Lambda(B(t^*))(t-t^*)+\int_{t^*}^t\psi(\tau)d\tau<0\ , $$
then $\|X(t)\|<\|X(t^*)\|$. 
Therefore the obtained  contradiction yields the stability for a trivial solution of the system of equations  (\ref{1}).
Theorem is proved. 
\end{proof}

Now we study the case when $f_k(0), k=1,2,...,n$ is not necessarily equals to zero. 
It is essential  to  use Definition 3. 

\begin{theorem}
Let for any $t,t^*\in [0,\infty), t^*\geq 0, t>t^*$ inequalities be fullfilled\\
1) $\Lambda(B(t))+\frac{1}{t-t^*}\int\limits_{t^*}^t \varphi^*(\tau) d\tau<0$,
where
$\varphi^*(s)=\int\limits_0^s\|C(s,\tau)\|d\tau+\frac{\|F'(s)\|}{\varepsilon}\le \int\limits_0^s\|C(s,\tau)\|d\tau+\frac{\delta_1(\varepsilon)}{\varepsilon}$,\\
2) $\delta_2(\varepsilon)\leq \varepsilon$,\\
and $\varepsilon(\varepsilon>0)$ is sufficiently small positive number. Then the trivial solution of the system (\ref{1}) is stable and $\|X(t)\|\leq \varepsilon, t\in [0,\infty)$.
\label{Theorem21}
\end{theorem}

\begin{proof}[Proof of Theorem 2]
The analysis  of stability is performed  similar to that one for $f_k(0)=0, k=1,2,...$. The difference is that we have to study the violation of  stability in the vicinity of the point $t=0$.

Let us consider this case. We show that when the conditions of the theorem are fulfilled,  
$\|X(t)\|\leq\varepsilon$ for $t\in[0,\infty)$. 

We will prove the assertion by contradiction. 
Let $\|X(0)\|=\varepsilon$, $\|X(t)\|>\varepsilon$ for $t>0$.

For $t>0$, the Cauchy problem  (\ref{7}), (\ref{8})  solution  is
\begin{equation}
	\begin{array}{lll}
		X(t)=\exp\{B(0)t\}X(0)+\int\limits_0^t \exp\{B(0) (t-s)\}(B(s)-B(0))X(s) ds+\\
		+\int\limits_0^t \exp\{B(0) (t-s)\} (\int\limits_0^s C(s,\tau)X(\tau) d\tau)ds+\int\limits_0^t \exp\{B(0) (t-s)\}F'(s) ds.
	\end{array}
	\label{v24}
\end{equation}

Repeating the foregoing argument for Theorem \ref{Theorem1}  we derived a contradiction. Therefore it follows  $\|X(t)\|\leq \varepsilon$ for $t\in[0,\infty)$.
\end{proof}
\section{Stability of solutions of systems of Volterra integro-differential equations}

Consider a system of Volterra integro-differential equations
\begin{equation}
\begin{array}{rcl}
\frac{dx_k(t)}{dt} & = & \sum\limits^n_{l=0} a_{kl} (t) x_l(t)+\sum\limits^n_{l=0} 
                      \int\limits^t_0  b_{kl} (t,\tau) x_l(\tau) d\tau, \\
 & & t\in [0,\infty), \ \ k,l=1,2,...,n, \\
 x_k(0)& = & x_k, \\ 
 & & k=1,2,...,n.   \\
\end{array}
\label{17}
\end{equation}

By analogy with the definition of stability of solutions  of ordinary differential equations systems
and following the work \cite{Crisci}, we provide the definition of stability of solutions of Volterra integro-differential equations systems.

{\bf Definition 5.} A trivial solution of Volterra integro-differential equations system (\ref{17}) is stable if for any small $\varepsilon>0$ there is $\delta(\varepsilon)$ so that the inequality
$\|X(t)\|\le \varepsilon$ follows from $\|X(0)\|\le \delta(\varepsilon)$. Here
$X(t)=(x_1(t), x_2(t),\ldots,x_n(t))$, $X(0)=(x_1(0), x_2(t),\ldots,x_n(t))$.

We  seek conditions for stability of a trivial solution of the system of equations (\ref{17}) assuming that:\\
1) for any arbitrary values  $x_k, k=1,2,...,n$ the Cauchy problem  (\ref{17}) has a unique solution;\\
2) the functions  $a_{kl}(t), b_{kl}(t,\tau), k=1,2,...,n$ are continuous. 

Write the Cauchy problem  (\ref{17}) in the matrix form
\begin{equation}
\frac{dX(t)}{dt}=A(t)X(t)+\int\limits_0^tB(t,\tau)X(\tau)d\tau,
\label{19}
\end{equation}
\begin{equation}
X(0)=X_0,
\label{20}
\end{equation}
where $A(t)=\{a_{kl}(t)\},  k,l=1,2,\ldots,n$; $B(t,\tau)= \{b_{kl}(t,\tau)\},  k,l=1,2,\ldots,n$;  \\ $X(t)=(x_1(t),\ldots,x_n(t)), x(0)=(x_1(0),\ldots,x_n(t))$.

Let us investigate stability of the trivial solution of the equation  (\ref{19}) with perturbed initial value.  
We will show that when for any $t_1, t_2>0\ $, $t_1<t_2$ the following condition is satisfied 
\begin{equation}
\Lambda(A(t))+\frac{1}{t_2-t_1}\int_{t_1}^{t_2}\int_{0}^{s}\|B(s,\tau)\|d\tau ds <0, \ t\in[0,\infty),
\label{21}
\end{equation}
the trivial solution of the system of equations (\ref{19}) is asymptotically stable.  

Show that when  (\ref{21}) is satisfied,  from   condition $\|x(0)\|\leq\delta, x(0)=x_0$,  the validity of  inequality $\|X(t)\|\leq\delta$   follows for $t \in [0,\infty)$. 

We will prove the assertion  by contradiction.  Assume that at the moment $t^*, t^*\ge 0$,  the trajectory of the solution of Cauchy problem (\ref{19}) with condition 
\begin{equation}
X(0)=X_0, \|X_0\|=\delta, 
\label{22}
\end{equation}
leaves a ball  $B(0,\delta)$. 

The solution of the Cauchy problems  (\ref{19}), (\ref{22}) for $t \in [t^*,\infty)$ can be written in the form  
\begin{equation}
\begin{array}{lll}
X(t)=\exp\{A(t^*)(t-t^*)\}X(t^*)+
\\
+\int_{t^*}^t\exp\{A(t^*)(t-s)\}(A(s)-A(t^*))X(s)ds+   \\
+\int_{t^*}^t\exp\{A(t^*)(t-s)\}\left(\int_0^sB(s,\tau)X(\tau)d\tau\right)ds.
\end{array}
\label{23}
\end{equation}

Proceeding to the norm we have  
\[
\begin{array}{lll}
\|X(t)\|\leq \exp\{\Lambda(A(t^*))(t-t^*)\}\|X(t^*)\|+
\\
+\int_{t^*}^t\exp\{\Lambda(A(t^*))(t-s)\}\|A(s)-A(t)\| \|X(s)\|ds+
\\
+\int_{t^*}^t\exp\{\Lambda(A(t^*))(t-s)\}\left(\int_0^s\|B(s,\tau)\| \|X(\tau)\|d\tau\right)ds.
\end{array}
\]

Repeating the arguments used to prove  Theorem \ref{Theorem1}  we can show that there is an interval $[t^*,t^{**}]$ such that for $t \in [t^*,t^{**}]$  the previous inequality can be strengthened 
\begin{equation}
\begin{array}{lll}
\|X(t)\|\leq \exp\{\Lambda(A(t^*))(t-t^*)\}\|X(t^*)\|+
\\
+\varepsilon_1\int_{t^*}^t\exp\{\Lambda(A(t^*))(t-s)\}\|X(s)\|ds+
\\
+\int_{t^*}^t\exp\{\Lambda(A(t^*))(t-s)\}\left(\int_0^s\|B(s,\tau)\|d\tau\right)\|X(s)\|ds=
\\
=\exp\{\Lambda(A(t^*))(t-t^*)\}\|X(t^*)\|+
\\
+\int_{t^*}^t\exp\{\Lambda(A(t^*))(t-s)\}g(s)\|X(s)\|ds, 
\end{array}
\label{24}
\end{equation}
where $g(s)=\varepsilon_1+\int_0^s\|B(s,\tau)\|d\tau$, $\varepsilon_1=\max_{s \in [t^*,t^{**}]} 
\|A(s)-A(t^*)\|$.

Introduce a function
\[\varphi(t)=\exp\{-\Lambda(A(t^*))t\}\|X(t)\|.
\]

The inequality  (\ref{24}) can be written in the form  
\begin{equation}
\varphi(t)\leq\varphi(t^*)+\int_{t^*}^tg(s)\varphi(s)ds. 
\label{25}
\end{equation}

Applying the Gronwall-Bellman inequality to (\ref{25}) we get
\[\varphi(t)\leq\varphi(t^*)  \exp\left\{\int_{t^*}^tg(s)ds\right\}.\]
So,
\[\|X(t)\|\leq\exp\left\{\Lambda(A(t^*))(t-t^*)+\int_{t^*}^tg(s)ds\right\}\|X(t^*)\|.\]

Hence, if for  $t \in [t^*,t^{**}]$ the following inequality is satisfied
\[\Lambda(A(t^*))(t-t^*)+\int_{t^*}^tg(s)ds<0,
\]
then the trajectory of the Cauchy problem  (\ref{17})  solution does not leave the ball $B(0,\delta)$. 

Thus, we have a contradiction which implies that when the inequality 
\[\Lambda(A(t))+\frac{1}{t-t^*}\int_{t^*}^{t} g(s)ds<0, \ \ \ t\in [0,\infty) \]
holds, stability of a trivial solution of the system of equations  (\ref{17}) is verified. 

It is easy to see that this inequality follows from the inequality (\ref{21}) when 
$ |t^{**}-t^* |$ are small enough. 

We note, when the inequality (\ref{21}) is fulfilled, the norm of the solution of equation (\ref{17}) $\|X(t)\|$ is monotonically decreasing.
Indeed, if the function $ \|X(t)\|$ at some point $\hat t$ yields a local minimum, then it is sufficient to repeat the foregoing arguments in order to demonstrate monotonic behavior.

We have, thus, proved the following theorem.
\begin{theorem}Let the conditions be fulfilled 
\[ 
\Lambda(A(t))+ \int_0^t \| B(t,s)\|ds < 0.
\]
Then the trivial solution of the system of equation (\ref{19}) is stable.  
\label{Theorem3}
\end{theorem}

\section{Numerical illustrations}
\begin{figure}[h!]
 \centering
 \begin{tabular}{cc}
   \begin{overpic}[width=0.5\linewidth]{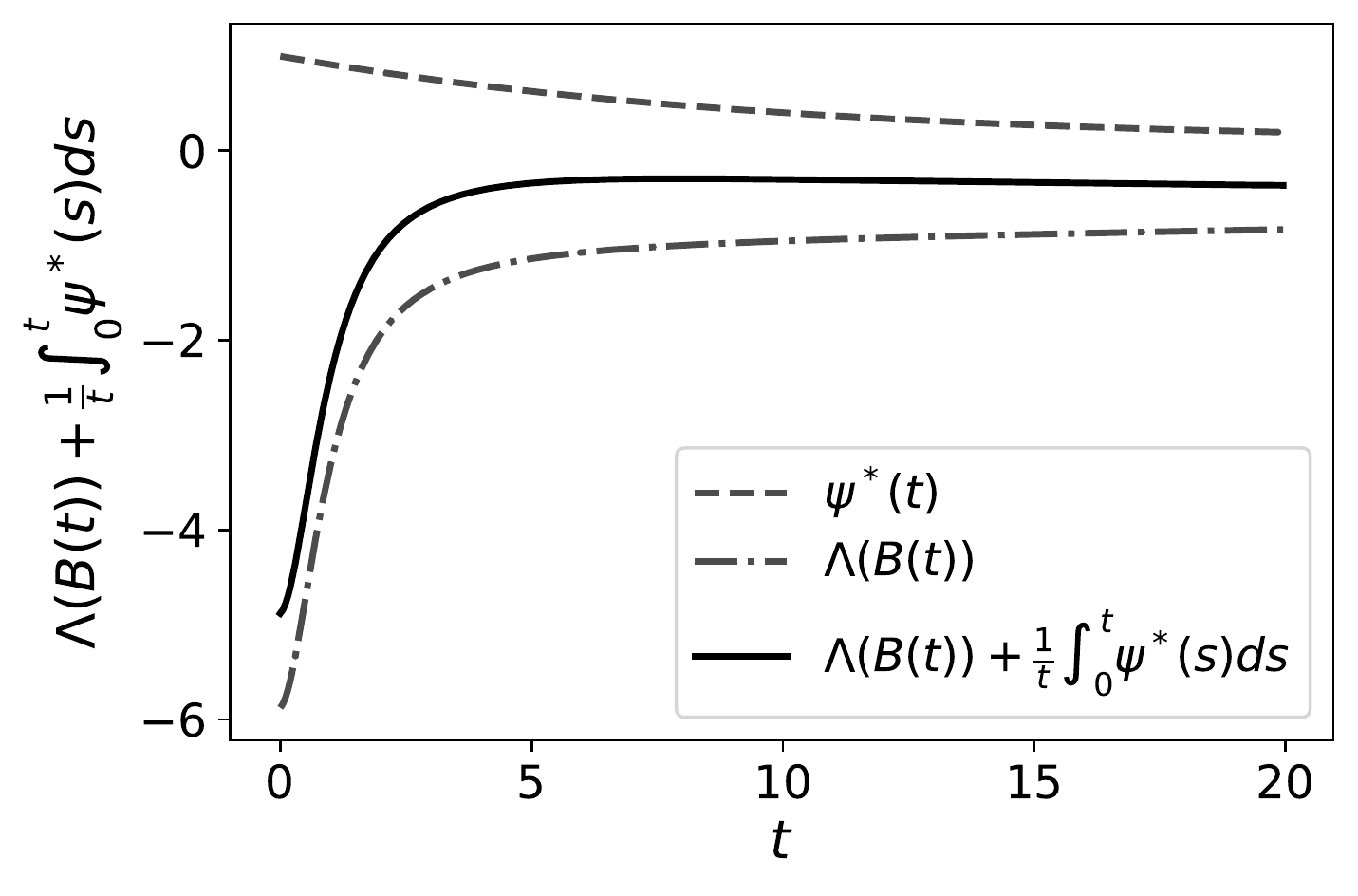}
   \put(22,45){a)}
   \end{overpic}
 &
   \begin{overpic}[width=0.5\linewidth]{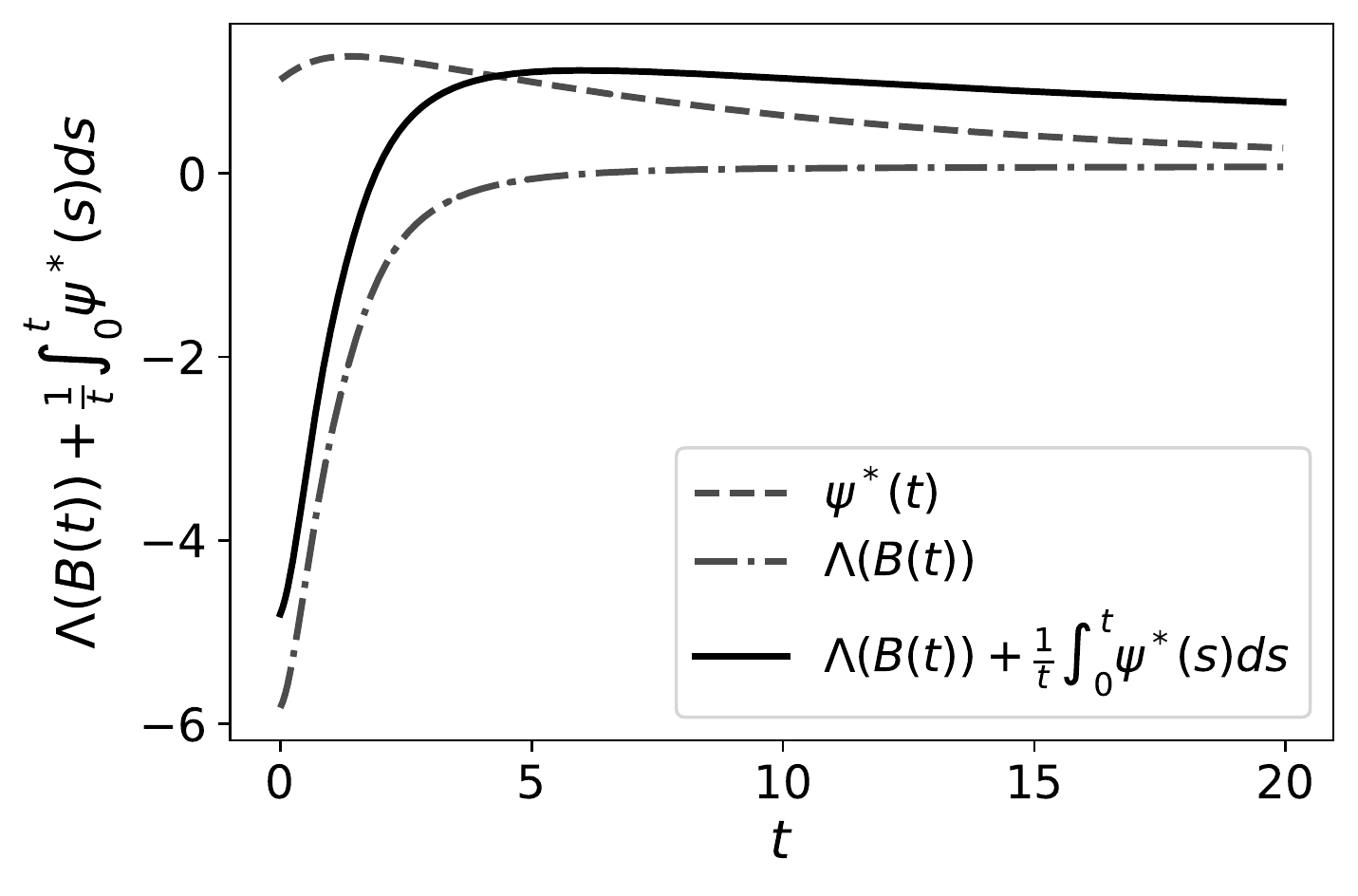}
    \put(20,45){c)}
   \end{overpic}
 \\
   \begin{overpic}[width=0.5\linewidth]{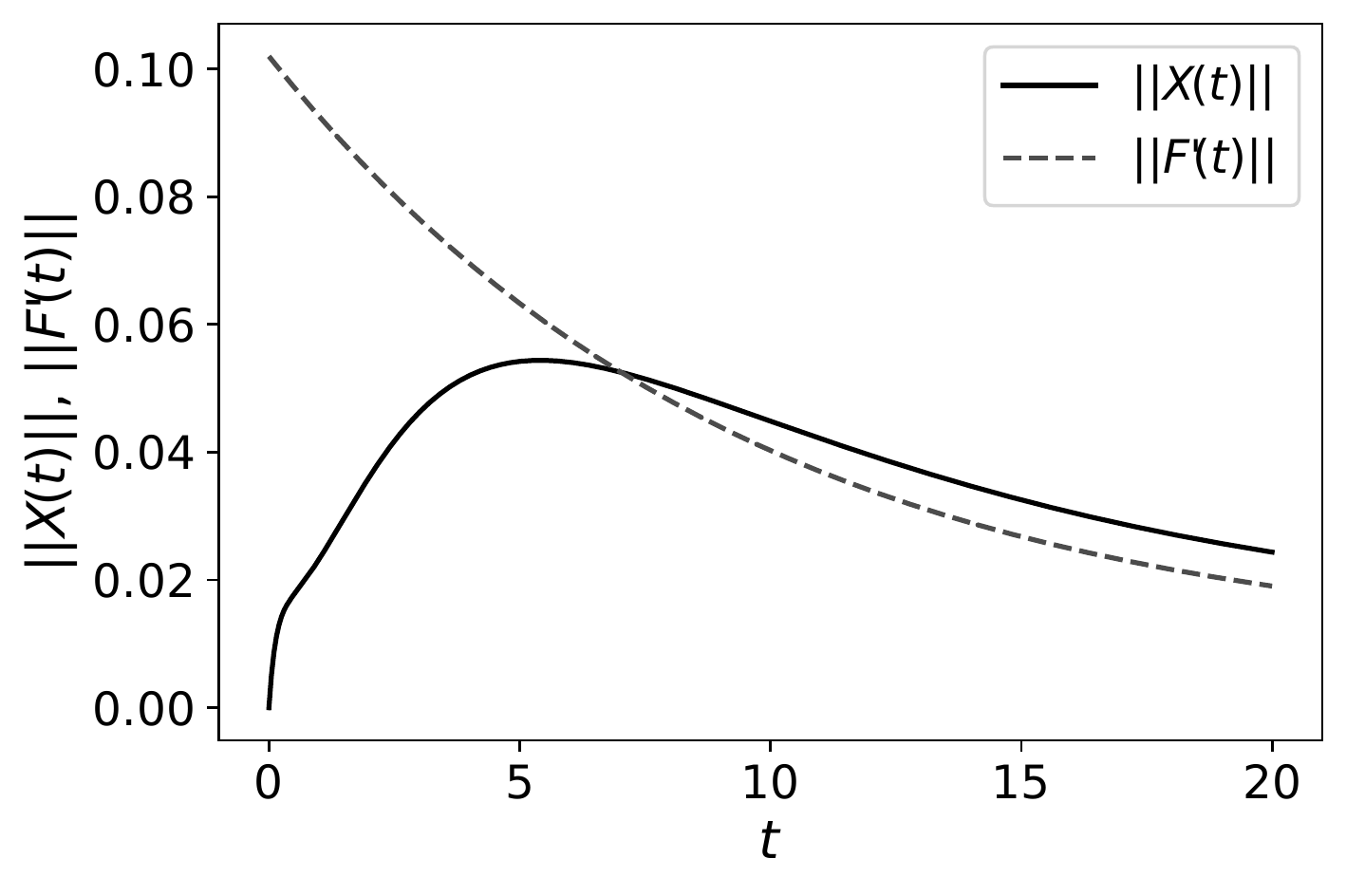}
    \put(20,45){b)}
   \end{overpic}
 &
   \begin{overpic}[width=0.5\linewidth]{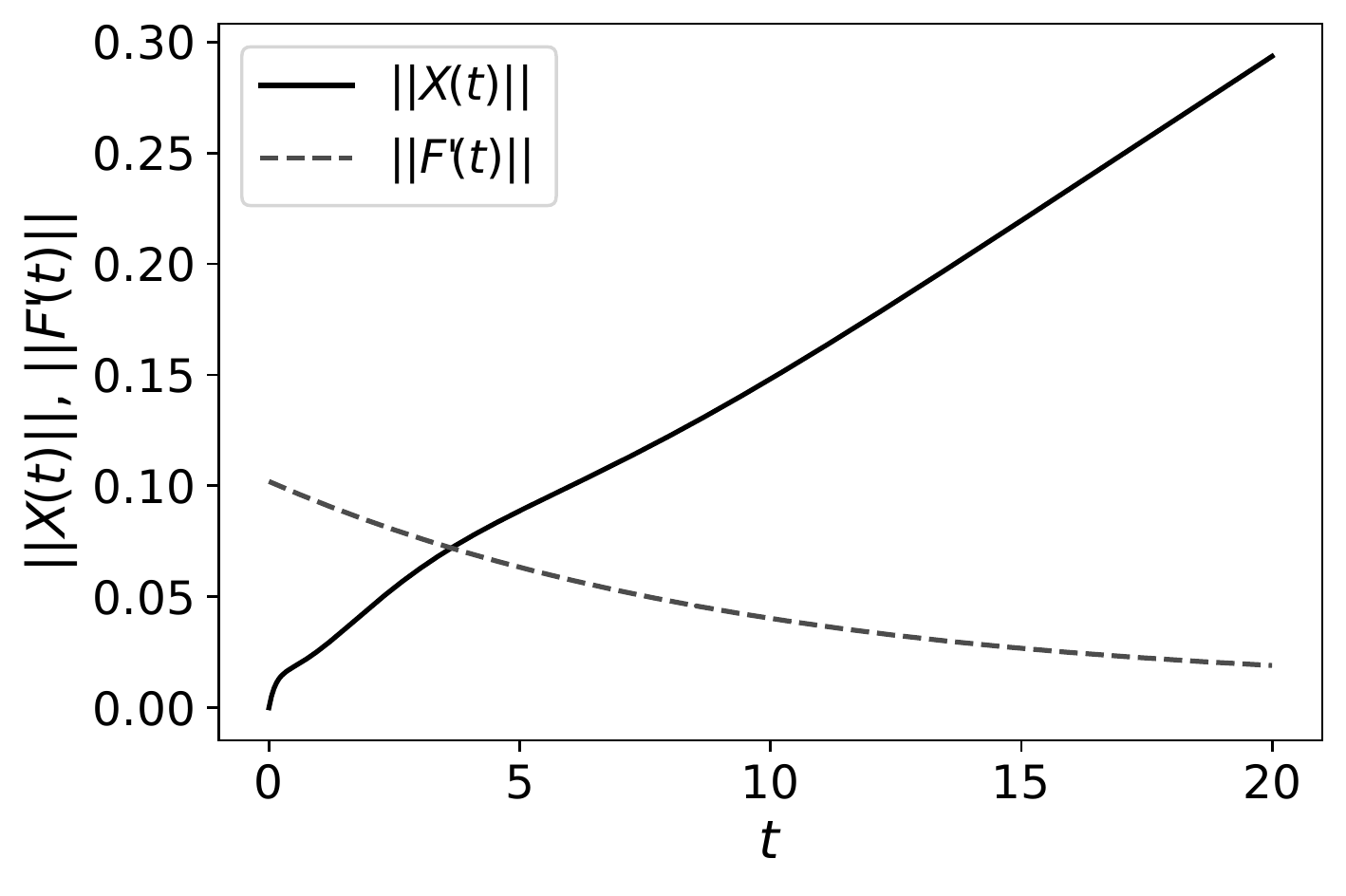}
    \put(20,45){d)}
   \end{overpic}
 \end{tabular}
 \caption{Stability of the solution for equation (\ref{7}) with respect to inhomogeneous term. 
 a),b) -- the conditions of Theorem \ref{Theorem1} are satisfied for the entire range of $t$.
 a) -- the stability indicating function (see Theorem \ref{Theorem1}) and its components; b) -- the norms of the solution and the inhomogeneus term, the norm of the solution does not exceed the supremum of the inhomogeneous term norm;
 c),d) -- the conditions of Theorem \ref{Theorem1} are satisfied only for a small subrange of $t$.
 c) -- the stability indicating function and its components; d) -- the norms of the solution and the inhomogeneus term, the norm of the solution exceeds the supremum of the inhomogeneous term norm when the conditions of Theorem \ref{Theorem1} break.
 \label{fig:th1} }
\end{figure}
In this section we illustrate theorems 1 and 3, i.e. stability of the solution with respect to the inhomogeneous term (Theorem \ref{Theorem1}) and stability of the trivial solution with respect to perturbations of the initial conditions (Theorem \ref{Theorem3}).

First, we solve equation (\ref{5}) by solving the associated Cauchy problem (\ref{7}). The operator and the inhomogeneous term read
\begin{equation}
  B(t,\tau)=\left(
            \begin{array}{cc}
             
               \alpha_{11}  \frac{e^{-\beta_{11}(t+\tau)}}{1+\tau^2}-\gamma e^{-\tau/100} 
               & \alpha_{12}  \frac{2^{-\beta_{12}t}}{1+\tau^3} \\
               \alpha_{21}  \frac{e^{-\beta_{21}(t+\tau^2)}}{1+\tau^4} 
               & \alpha_{22}  \frac{3^{-\beta_{22}\tau}}{1+\tau^2}-\gamma e^{-\tau/100} 
             
            \end{array}
            \right) \ ,
 \label{eq:th1numer}
\end{equation}
\begin{equation}
   F(t)=\left(
        \begin{array}{c}
          e^{-t/10}-1 \\
          2 e^{-t/50}-2
        \end{array}
        \right) \ .
  \label{eq:th1inhom}
\end{equation}
We use two sets of parameters of the operator.
First, we choose the parameters $\alpha_{ik}$ and $\beta_{ik}$ so that the conditions of the Theorem \ref{Theorem1} are satisfied for the whole range of $t\in[0,20]$. We use the following values:

\[
  \alpha_{ik}=\left(
              \begin{array}{cc}
               -5 & 1\\
               \frac{1}{2} & -1
              \end{array}
              \right) \ ,
\]
\[
  \beta_{ik}=\left(
              \begin{array}{cc}
                0,001 & 0,001\\
               0,002 & 0,001
              \end{array}
              \right) 
\]
and $\gamma=1$.
Second,  we modify the parameters $b_{ik}$ and $\gamma$ so that the conditions of the theorem are satisfied only in the beginning of the interval $t\in [0,t^*]$. In this case the following values have been used:
\[
  \beta_{ik}=\left(
              \begin{array}{cc}
                0,1 & 0,1\\
               0,2 & 0,1
              \end{array}
              \right) 
\]
and $\gamma=0$.

We solve equation (\ref{eq:th1numer}) using a standard explicit Euler method with uniform step $\Delta t$ and the trapezoidal rule for integral evaluation. Namely, at each time step $t_i=i \Delta t$
we evaluate the integral operator as
\[
  Y_i\equiv \int_0^{t_i} B'_t(t_i,\tau) X(\tau) d\tau =\Delta t \sum_{k=0}^{i} B'_t(t_i,t_k) X_k 
  -\frac{\Delta t}{2}(B'_t(t_i,t_0) X_0+B'_t(t_i,t_i) X_i)+O(\Delta t^2)
\]
and the trajectory is being updated
\[
  X_{i+1}=X_{i}+\Delta t (B(t_{i},t_{i}) X_{i}+Y_{i}+F'(\frac{t_i+t_{i+1}}{2})) \ .
\]
Here $B'_t(t,\tau)$ stands for the partial derivative $B'_t(t,\tau)=\frac{\partial B(t,\tau)}{\partial t}$ and $F'(t)$ stands for the derivative of the inhomogeneous term (\ref{eq:th1inhom}) (see equation (\ref{7})).
The time step $\Delta t=0.025$ is sufficient to calculate the solution converged within the resolution of the figure for $t \in [0,20]$. 

The numerical results are shown in figure~\ref{fig:th1}. When the conditions of Theorem \ref{Theorem1} are satisfied everywhere within the given range, the norm of the solution does not exceed the supremum of the norm of the inhomogeneous term (figure~\ref{fig:th1}a and \ref{fig:th1}b). Contrary, when the conditions of the theorem break  (figure~\ref{fig:th1}c and \ref{fig:th1}d), the norm of the solution keeps growing.

\begin{figure}[h!]
 \centering
    \includegraphics[width=1.0\linewidth]{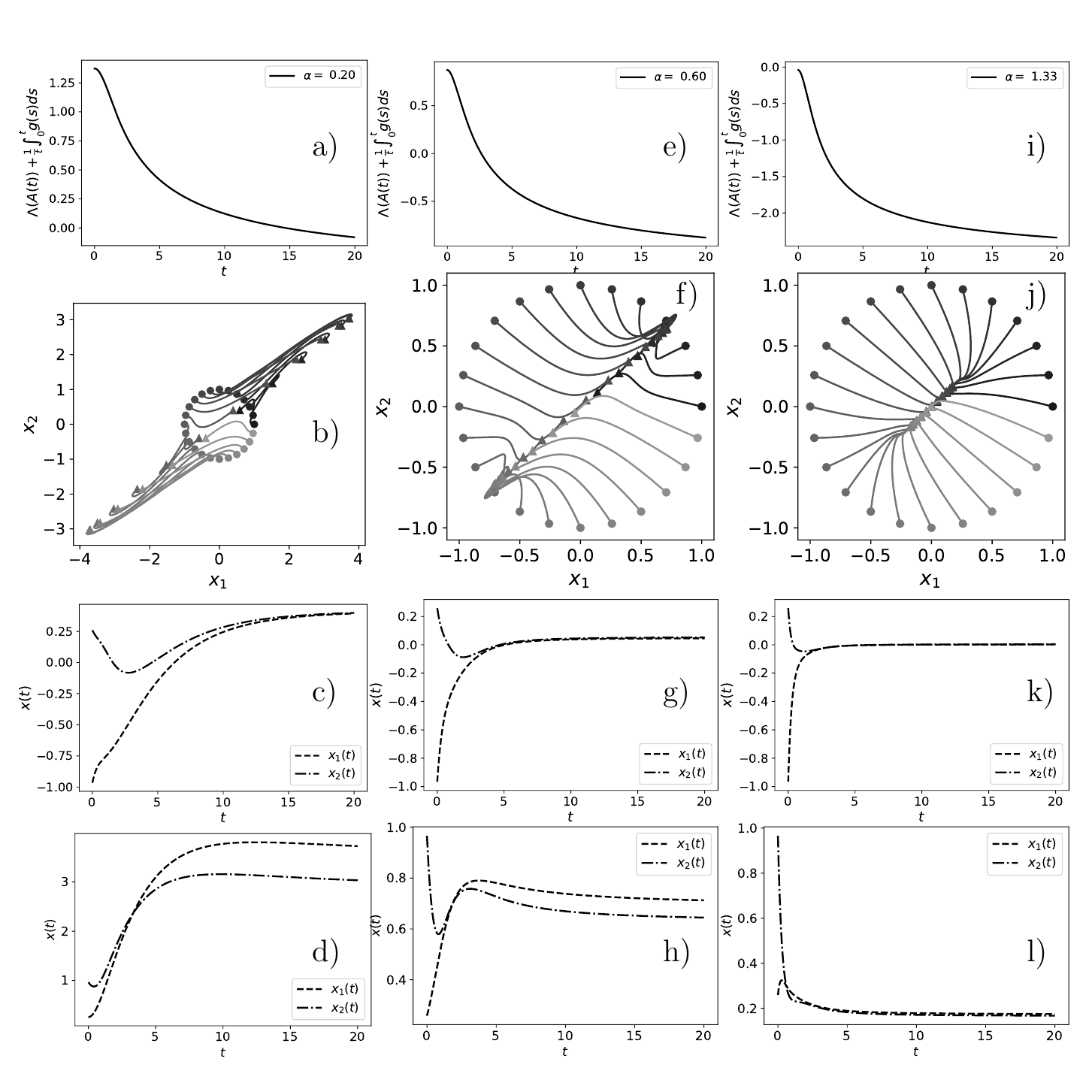}
 \caption{Stability of the trivial solution for equation (\ref{eq:example}). 
 a),e),i) -- the stability indicating function (\ref{21}) for different values of $\alpha$; b),f),j) -- the solutions starting on a unit sphere  for different values of $\alpha$;
 c),g),k) -- the solutions with minimal deviation from the trivial solution for different values of $\alpha$;
 d),h),l) -- the solutions with maximal deviation from the trivial solution for different values of $\alpha$. The values of $\alpha$ are as indicated in subfigures a),e),i): 
 a)-d) $\alpha=0.2$; e)-h) $\alpha=0.6$; i)-l) $\alpha=1.33$.
 \label{fig:fig1} }
\end{figure}
In order to illustrate Theorem \ref{Theorem3} we check for stability of the trivial solution for the following equation
\begin{equation}
   \frac{dX(t)}{dt}=\alpha A(t)X(t)+\int_0^t H(t,\tau) X(\tau) d\tau \ .
\label{eq:example}
\end{equation}
Here $a(t)$ is a linear operator defined by its matrix
\[
  A(t)=\left(
       \begin{array}{cc}
         -2 & \frac{1}{1+t^2}\\
          \frac{1}{2+t^2} & -2
       \end{array}
  \right) \ ,
\]
and the integral kernel is taken as
\[
  H(t,\tau)=\left(
       \begin{array}{cc}
         \frac{1}{1+t^2+\tau^2} & \frac{1}{1+\tau^3} \\
          \frac{t}{2+t^3+\tau^4} & \frac{1}{1+\tau^4} 
       \end{array}
  \right) \ .
\]
Here parameter $\alpha$ in the equation (\ref{eq:example}) regulates the balance between the negative logarithmic norm of the operator $A$ and the norm of the integral operator. We shall discuss it in more details later on.

As the solutions of the equation (\ref{eq:example}) are scale-invariant, it is sufficient to track the trajectories that originate from the unit circle. 

We used a numerical scheme  similar -- up to the notation and the absence of the inhomogeneous term -- to the one employed in the illustration for Theorem \ref{Theorem1}. 

We use parameter $\alpha$ in equation (\ref{eq:example}) to illustrate three different regimes. For small $\alpha$ the stability indicating function in Theorem \ref{Theorem3}  -- see also equation (\ref{21}) -- is dominated by the integral operator and is strictly positive in the interval (see figure~\ref{fig:fig1}a). In this case the stability of the trivial solution is not guaranteed. The solutions originating from the unit circle are shown in figure~\ref{fig:fig1}b. We see that the most of the trajectories leave the unit circle rapidly, while only a few stay within (figure~\ref{fig:fig1}c).  For an intermediate value of $\alpha$ the stability indicating function changes its sign (figure~\ref{fig:fig1}e), and the stability is still not guaranteed. In our case the solution stays roughly within the range of the initial condition, not approaching the trivial solution(figure~\ref{fig:fig1}f). For $\alpha=1.33$ the stability indicating function  is strictly negative (figure~\ref{fig:fig1}i),
which warrants the stability of the trivial solution. Indeed, all the trajectories end closer to the origin than they start (figure~\ref{fig:fig1}j), which we also see tracking the solutions of the minimal and the maximal deviation from the trivial solutions (figure~\ref{fig:fig1}k,l). The both trajectories approach some small constants asymptotically. We should emphasize, that even though the trivial solution is not an attractor for the equation (\ref{eq:example}), it is still a stable solution, as all the solutions starting in its vicinity do stay there.


\section{Conclusions}

In this paper we have proposed a method to obtain sufficient conditions for stability  of solutions of Volterra linear integral equation systems and systems of
Volterra linear integro-differential equations.  The method is based on computing logarithmic norms of matrices consisting of coefficients of equations and derivatives of their kernels. 
This method can also be used to obtain sufficient conditions for stability of solutions of systems of other Volterra integral equations, besides (\ref{1}) or (\ref{19}).

The authors are considering the following directions for future research: obtaining sufficient stability conditions for nonlinear Volterra equations as well as the  equations with fractional derivatives and integrals. We are also going to apply our results to studying of immune response processes in living organisms modeled by the systems of Volterra integro-differential equations.


\vspace{6pt} 



\end{document}